\documentclass[a4paper,10pt]{article}
\usepackage{amsfonts, amsmath, amssymb, amsthm, amscd, hyperref,graphicx}
\usepackage{color}
\usepackage{enumerate}

\usepackage{a4wide}
\usepackage{nicefrac}

\newtheorem{theorem}{Theorem}
\newtheorem{problem}{Problem}

\newtheorem{lemma}{Lemma}

\theoremstyle{definition}
\newtheorem{defn}{Definition}

\theoremstyle{remark}
\newtheorem*{rem}{Remark}

\newcommand{\calP}{{\cal P}}
\newcommand{\calQ}{{\cal Q}}

\newcommand{\calZ}{{\cal Z}}

\newcommand{\Q}{{\mathbb Q}}
\newcommand{\R}{{\mathbb R}}
\newcommand{\Z}{{\mathbb Z}}

\def\e{\varepsilon}
\def\eps{\varepsilon}
\def\x{\bf{x}}
\def\y{\bf{y}}

\newcommand{\commentNICK}[1]{{\color{red} #1}}

\newcommand{\commentSINAI}[1]{{\color{green} #1}}

\title{Convex curves and a Poisson imitation of lattices}
\author{Nick Gravin, Fedor Petrov, Dmitry Shiryaev and Sinai Robins}

\begin{document}

\maketitle
\begin{abstract}
We solve a randomized version of the following open question:   is there a strictly convex, bounded curve 
$\gamma \subset \mathbb R^2$ such that the number of rational points on $\gamma$, with denominator $n$, approaches infinity with $n$?  Although this natural problem appears to be out of reach using current methods, we consider a probabilistic analogue using a spatial Poisson-process that simulates the refined rational lattice $\frac{1}{d}\mathbb Z^2$, which we call $M_d$, for each natural number $d$.   The main result here is that with probability $1$ there
exists a  strictly convex, bounded curve $\gamma$ such that
$
|\gamma\cap M_{d}|\rightarrow  +\infty,
$
as $d$ tends to infinity.   The methods include the notion of a generalized affine length of a convex curve, defined in  \cite{P2}.

\end{abstract}

\section{Introduction}

We first recall a natural and as yet unsolved problem from  \cite{P2}, that
arose in the context of the geometry of numbers:

\begin{problem}\label{motivation}
Does there exist a strictly convex bounded curve $\gamma$
in the plane such that, as $n$ tends to infinity,
\begin{equation}\label{eq:lattice}
\left|\gamma\cap \frac1n \mathbb{Z}^2\right|  \rightarrow \infty?
\end{equation}
\end{problem}

The problem we address here is a probabilistic
analogue of problem \eqref{motivation}, for which we first give some motivation and definitions.
We recall from \cite{P2} the following results, giving lower bounds for rational points on any strictly 
convex curve $\gamma$ in the plane:

\begin{align}
\left|\gamma\cap \bigcup_{k=1}^n\frac1k \mathbb{Z}^2\right|&=O(n), \\
\sum_{k=1}^n \left|\gamma\cap \frac1k \mathbb{Z}^2\right|&=O(n\log n).
\end{align}

The latter lower bound means that in the case of a positive answer to problem \ref{motivation},
the number of points of $ \frac1n \mathbb{Z}^2$ on $\gamma$ cannot converge to infinity
faster than $\log n$.  We also recall  that in higher dimensions there is no gap in such asymptotic estimates. Namely, it is proved in \cite{P2} that for a bounded, closed, and strictly convex surface 
$\Gamma$ in $\mathbb{R}^d$, with $d\geq 3$, we have
$$
\liminf \frac{|\Gamma\cap n^{-1}\mathbb{Z}^d|}{n^{d-2}}<\infty.
$$

As added motivation, we note that there are probabilistic models of rigid arithmetical objects,
like the integers or the primes, which capture many
quantitative characteristics of
the original rigid structure.   These probabilistic analogues are often easier to handle and offer support for the validity of their  corresponding ``rigid'' statements.
For instance, the Riemann Hypothesis may be considered as a statement
that the M\"obius $\mu$-function displays some quasi-random behavior. Specifically,
the following asymptotic for the partial sums of the M\"{o}bius function $\mu(n)$ is known to be equivalent to RH:
$\sum_{n\leq x} \mu(n)=o(x^{1/2+\eps})$
for any $\eps>0$.  However, if we replace the deterministic sequence $\{ \mu(1),\mu(2),\dots \}$
by independent identically distributed
bounded variables with zero mean, then such an estimate  holds with probability $1$,
by standard methods.
Because there is no evidence that the values of the M\"{o}bius function are highly dependent,
this asymptotic is considered as heuristic support for the Riemann Hypothesis.

Here we deal with another arithmetical object,
the integer lattice $\mathbb{Z}^d$ in $d$-dimensional Euclidean space.  We consider
its natural probabilistic analogue: a random configuration of points given by a Poisson process with
intensity $1$ in $\mathbb{R}^d$.

As a non-example for problem \eqref{motivation}, consider the unit circle $S$, centered at the origin.
 Then $\frac1n \sum_{k=1}^n |S\cap n^{-1}\mathbb{Z}^2|$
tends to infinity. In other words, on average the unit circle has many rational points.
However, for any prime $n$ the unit circle has only at most 12 rational points
with denominator $n$, and hence the circle does not satisfy
property \eqref{eq:lattice}.   Moreover, if we consider any real algebraic curve $\gamma$ in the plane, of genus $g \geq 2$,  we know by Faltings' theorem that it contains only finitely many rational points.   Hence all such curves $\gamma$ do not satisfy \eqref{eq:lattice}.
Though we do not know of any formal argument that allows us to resolve problem \ref{motivation},
these counterexamples suggest a negative answer to this problem.

On the other hand, ``most'' convex curves are not algebraic, and if we pick a convex curve $\gamma$
``at random'', it is natural to ask if $\gamma$ might still satisfy
property \eqref{eq:lattice}?

In sharp contrast with algebraic curves, it turns out that if we work in the probabilistic framework suggested above, then we find
here an affirmative answer to a very natural probabilistic analogue of problem \ref{motivation}.

To describe the main result, we first order all of the prime powers in ascending order, as follows:
$q_1 = 2, q_2 = 3, q_3 = 2^2, q_4 = 5, q_5 = 7,
q_6=2^3, q_7=3^2, q_8 = 11, q_9 = 13, q_{10} = 2^4$, and so on.   For technical reasons that are explicated in
the preliminaries section \ref{preliminaries} below,
we define $w_{q_k}:= q_k^2 \left( 1- \frac{1}{p^2} \right)$, where $q_k$ is a power of the prime $p$.  The value $w_{q_k}$ will be the intensity of our Poisson process.

We define our randomized analogue of a lattice
to be a Poisson point set $M_{q_k}$, with intensity equal to $w_{q_k}$, for each $q_k$ which is a power of a prime.
  In other words, the randomized point set $M_{q_k}$ by definition satisfies a spatial Poisson process distribution.   It will turn out that the only property we really require in our analysis is that the intensity $w_{q_k} > \frac{q_k^2}{2}$.
In  section \ref{preliminaries}, we define a more general Poisson analogue of a lattice, called $M_d$ for any integer $d$, but
we show that due to the elementary combinatorial
structure of $M_d$, it is sufficient to work with the particular analogue $M_{q_k}$ defined here.


\begin{theorem}\label{tm:main} With probability $1$ there
exists a  strictly convex, bounded curve $\gamma$
such that
$$
|\gamma\cap M_{q_k}|\rightarrow  +\infty,
$$
as $k$ tends to $\infty$.
\end{theorem}

Many geometric questions regarding lattices appear to have similar answers to
their corresponding Poisson framework questions, in the context of a `perturbed lattice'.
Here are some more examples.


Let $K$ be a convex and compact set in the plane. We consider
all possible convex polygons with vertices in $\frac{1}{n}\mathbb{Z}^2\cap K$ for
some large $n$.  We may now ask some natural and intuitive questions:

\begin{problem}\label{one} how many such polygons are there in total?
\end{problem}

\begin{problem}\label{two} How many vertices does such a polygon have?
\end{problem}

\begin{problem}\label{three} What does a typical polygon look like?
\end{problem}

The same questions may be asked for a  random set of points
instead of points chosen from a lattice. For example, we may consider $C(K)\cdot n^2$
independetly and uniformly distributed points in $K$,
for some constant $C(K)$, independent of $n$.
We may also consider a Poisson process inside $K$, with intensity $n^2$.

The answers to above questions \eqref{one}, \eqref{two}, and \eqref{three} 
appear to be very similar in the lattice setting and in the probabilistic setting  (see \cite{V}, \cite{BV}, \cite{B1}).  Amazingly, only the specific values of constants differ. However,  the methods
of the geometry and numbers (lattice setting) and of stochastic
geometry (randomization) do indeed differ.
In this paper we work with a randomized model for the set of rational points $\Q^2$.

\section{Preliminaries}\label{preliminaries}

First, we use the Poisson process of intensity
$1$ in $\R^2$, which we call  $\calZ$,   as our most natural realization for a
randomized relaxation of the integer lattice $\mathbb Z^2$.
 In general, we denote by $\calP^2(n)$ a Poisson process of intensity $n$.
  We recall the definition of this spacial process. For more information, see, e.g. \cite{DVJ}

\begin{defn}
\label{def:poisson}
A Poisson process of intensity $\lambda$ is
characterized by the following two properties:
\begin{itemize}
\item For any region $A$ with area  $| A |$, the
number of events in $A$ obeys a $1$-dimensional
Poisson distribution with
mean $\lambda | A |$.
\item The number of events in any finite
collection of non-overlapping regions are independent of each other.
\end{itemize}
\end{defn}

In order to extend Poisson process to the rational set $\Q^2$ one might first represent
the original lattice by the union $\cup_{n=1}^{\infty}\frac{1}{n}\Z^2$ and then define the
randomized rational lattice as $\cup_{n=1}^{\infty}\frac{1}{n}\calZ^2$. Unfortunately, in this case
the probabilistic object does not enjoy the set-theoretical property of the original lattice,
namely that $\frac{1}{n}\Z^2\subset\frac{1}{m}\Z^2$ for $n|m$.
For this reason we introduce a slightly different randomization process.
We begin with a disjoint decomposition of $\Q^2$ given as follows.

\begin{equation}
\begin{array}{c}
\Q^2=\bigcup\limits_{n=1}^{\infty} L_n', \text{ where}\\
L_n':=\{(a/n,b/n),gcd(a,b,n)=1\}.
\end{array}
\end{equation}

We notice that each set $L_n'$ can be expressed by inclusion-exclusion principle
as a sum (with signs) of scaled integer lattices. In particular, if $p_1,\dots,p_k$
are all prime divisors of $n$, then
\[L_n'=\frac{1}{n}\Z^2-\frac{1}{\nicefrac{n}{p_1}}\Z^2-\dots-\frac{1}{\nicefrac{n}{p_k}}\Z^2+
\frac{1}{\nicefrac{n}{p_1p_2}}\Z^2+\frac{1}{\nicefrac{n}{p_1p_3}}\Z^2+\dots+\frac{(-1)^{k}}{\nicefrac{n}{p_1p_2\dots p_k}}\Z^2 .
\]
Thus in the randomization process it is natural to assume that points in each $L_n'$ are distributed
uniformly in the plane with the following density per unit area
$$
w_n:=n^2\prod_{\substack{p|n,\\ p\, is\, prime}}(1-1/p^2).
$$

We notice that $\frac1n\mathbb{Z}^2=\bigcup\limits_{d|n} L_d'$.
We define $M_n$ as a random Poisson configuration of intensity $w_n$.
Then the random analogue of $\Q^2$ is
$$
\calQ^2:=\bigcup_{n=1}^{\infty} M_n
$$
(processes with different $n$ are mutually independent).
We observe a few useful properties of such a process.
\begin{itemize}
 \item $\bigcup\limits_{d|n} M_d$ coincide with the usual Poisson process of
       intensity $n^2$, which is the standard randomized analog for the
       rational lattice  $\frac{1}{n}\Z^2$ with fixed denominator $n$.
 \item $\bigcup\limits_{d|n_1} M_d\cap\bigcup\limits_{d|n_2} M_d= \bigcup\limits_{d|\gcd(n_1,n_2)} M_d$
       and similarly for lattices $\frac{1}{n_1}\Z^2\cap\frac{1}{n_2}\Z^2=\frac{1}{\gcd(n_1,n_2)}\Z^2$.
       In this sense, the family of $\bigcup\limits_{d|n} M_d$ obeys the same set theoretical structure
       as the rational lattices.
 \item $\bigcup\limits_{k=1}^{n} M_k$ is also a Poisson process of certain intensity;
       its expected number of points coincides with the number of rational points
       with denominators not exceeding $n$ in any unit square of general position, i.e.
       $\bigcup\limits_{k=1}^{n} M_k$ is a natural analog of $\bigcup\limits_{k=1}^{n}\frac{1}{k}\Z^2$.
\end{itemize}

\section{Results about generalized affine length}

The proof uses the notion of a \textit{generalized
affine length of a convex chain} introduced in \cite{P1}.
 Let $S(F)$ denote the doubled area of a polygon
$F$, $\x\times\y$ denote the pseudo-scalar product
of vectors $\x$ and $\y$
(i.e. an oriented area of a parallelogram,
based on these vectors).

Fix a triangle $ABC$ oriented so that
$S=S(ABC)=+\overline{AC}\times \overline{CB}$.

Let $ \gamma = A C_1 C_2 \dots C_k\kern1pt B$
be a strictly convex chain.
We call it $(AB,C)$-chain, if all its
vertices lie inside triangle
$ABC$.

Let $(AB,C)$-chain
$\gamma=AC_1C_2\dots C_kB$ be inscribed in an
$(AB,C)$-chain $\gamma_1=AD_1D_2\dots D_{k+1}B$
(i.e. points $C_i$ lie on respective segments
$D_iD_{i+1}\,(i=1,\,2,\,\dots,\,k)$).
Define a \textit{generalized affine length} of a chain
$\gamma$ with respect to $\gamma_1$ as
\begin{equation}\label{genafflength}
l_A(\gamma:\gamma_1):=
\sum_{i=0}^{k} S(C_iD_{i+1}C_{i+1})^{1/3}\qquad(C_0=A,\;C_{k+1}=B),
\end{equation}

Part (i) of the following lemma is well-known, and part
(ii) is a non-surprising quantitative improvement
of (i).

\begin{lemma}\label{geom} (i) Let points
$P,\,R$ be chosen on the sides $AC$ and
$BC$ of the triangle $ABC$ respectively, and a point
$Q$ --- on the segment $PR$.
Then
$$
S(AQP)^{1/3}+S(BQR)^{1/3}\le S^{1/3}.
$$

(ii) For fixed $0<\alpha<S^{1/3}$ call a point
$Q$ inside $\triangle ABC$ $\alpha$-admissible if
there exist points $P$, $R$ on $AC$, $BC$ respectively
such that $Q$ lies on a segment $PR$ and
$$
S^{1/3}-S(AQP)^{1/3}-S(BQR)^{1/3}\leq \alpha.
$$
Then the area of the set of $\alpha$-admissible
points is not less than $\frac1{8}\sqrt{\alpha}S^{5/6}$.
\end{lemma}

\begin{proof}
 We denote
$$\e=\frac{\alpha}{S^{1/3}}, \text{   } 0<\e<1,$$ and
$$\rm{Err}=1-\bigg(\frac{S(APQ)}S\bigg)^{1/3}-
\bigg(\frac{S(BQR)}S\bigg)^{1/3}.$$
Our goal is to show that
\begin{enumerate}[(i)]
\item $\rm{Err}\geq 0$, and
\item Set of points $Q$, for which $\rm{Err}\leq \e$, has area at least $\frac1{8}\sqrt{\alpha}S^{5/6}=\frac1{8}\sqrt{\e}S$.
\end{enumerate}
By simple geometrical observations,

\begin{align}
\rm{Err}&=1-\bigg(\frac{S(APQ)}S\bigg)^{1/3}-
\bigg(\frac{S(BQR)}S\bigg)^{1/3}\notag\\
&=\frac{1}{3}+\frac{1}{3}+\frac{1}{3}-\bigg(\frac{AP}{AC}\cdot\frac{PQ}{PR}\cdot\frac{RC}{BC}\bigg)^{1/3}-
\bigg(\frac{PC}{AC}\cdot\frac{QR}{PR}\cdot\frac{BR}{BC}\bigg)^{1/3}\notag\\
&=\bigg(\frac13\bigg(\frac{AP}{AC}+\frac{PQ}{PR}+\frac{RC}{BC}\bigg)-
\bigg(\frac{AP}{AC}\cdot\frac{PQ}{PR}\cdot\frac{RC}{BC}\bigg)^{1/3}\bigg)+\notag\\
&\qquad+\bigg(\frac13\bigg(\frac{PC}{AC}+\frac{QR}{PR}
+\frac{BR}{BC}\bigg)-\bigg(\frac{PC}{AC}\cdot\frac{QR}{PR}
\cdot\frac{BR}{BC}\bigg)^{1/3}\bigg)\label{AMGM_sum}
\end{align}

Later is the sum of two expressions of the form
$(x+y+z)/3-(xyz)^{1/3}$. Hence the proof
of (i) is finished by
applying AM-GM inequality.

For the proof of (ii) we want first to estimate both of the aforementioned expressions in \eqref{AMGM_sum} from above.
We will do it in a particular case when all ratios involved in \eqref{AMGM_sum} lie between $1/8$ and $1$ and differ pairwise by at most $2\delta$,
where $\delta$ is a parameter, value for which we will assign later.

We note that for any non-negative numbers $x, y, z$ the following identity holds:
\begin{align}
&(x+y+z)/3-(xyz)^{1/3}=\notag\\
\label{AMGM_expansion}
=&\frac{1}{6}(x^{1/3}+y^{1/3}+z^{1/3})((x^{1/3}-y^{1/3})^2+(y^{1/3}-z^{1/3})^2+(z^{1/3}-x^{1/3})^2).
\end{align}

Also, when $x,y,z$ lie between $1/8$ and $1$ and differ pairwise by at most $2\delta$,
\begin{align*}
&\left|{x^{1/3}-y^{1/3}}\right|
=\frac{\left|x-y\right|}{\left|x^{2/3}+y^{2/3}+x^{1/3}y^{1/3}\right|}
\leq\frac{2\delta}{3(\frac{1}{8})^{2/3}}
=\frac{8\delta}{3},
\end{align*}
and we can use this to estimate right hand side of \eqref{AMGM_expansion} from above:
\begin{align*}
&\frac{1}{6}(x^{1/3}+y^{1/3}+z^{1/3})((x^{1/3}-y^{1/3})^2+(y^{1/3}-z^{1/3})^2+(z^{1/3}-x^{1/3})^2)\leq\\
\leq &\frac{1}{6}\cdot 3\cdot 3(\frac{8\delta}{3})^2 = \frac{32}{3}\delta^2
\end{align*}

Now, fix $\delta=\sqrt{\e}/8$. It immediately follows that
$$Err\leq 2\frac{32}{3}\delta^2\leq 2\cdot \frac{32}{3}\cdot (\frac{\sqrt{\e}}{8})^2\leq \e,$$
in other words, if all ratios in \eqref{AMGM_sum} lie between $1/8$ and $1$ and differ pairwise by at most $2\delta=\sqrt{\e}/4$,
then $Q$ is $\e$-admissible. It remains to prove that
the locus of such points $Q$ has area at least $\delta S=\frac1{8}\sqrt{\e} S$.

The statement of (ii) is preserved under affine
transforms, so without loss of generality, suppose $A=(-1,1)$, $B=(1,1)$, $C=(0,-1)$, and therefore $S=2$.
We will now produce a set of points $Q$ of area $\delta S=2\delta$, for which desired conditions on ratios of segments are satisfied.

\begin{figure}
\centering
\label{fig:triangle}
\def\svgwidth{150pt}
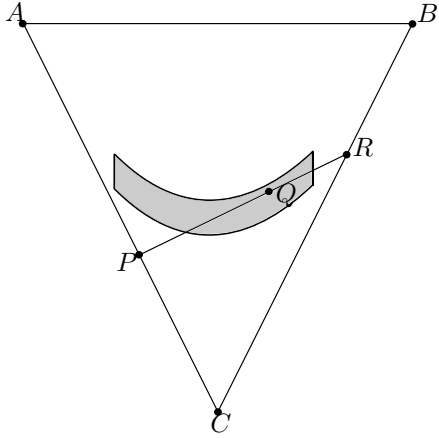
\caption{The shaded area consists of some $\e$-admissible points inside a triangle $ABC$.}
\end{figure}

Consider the point $Q=(t,t^2+\tau)$, $-1/2\leq t\leq 1/2$, $-\delta\leq \tau\leq \delta$.
Draw a line $y=2tx-t^2+\tau$ through $Q$ and let it meet sides $AC$, $BC$ in points $P,\,R$ respectively.
By choice of $\delta=\sqrt{\e}/8\leq 1/8$, all such points $Q$ indeed lie inside triangle $ABC$.
The locus of all such points $Q$ has area exactly $2\delta$.
Also with this choice of $t$ and $\tau$,
points $P$ and $R$ lie on sides $AC$ and $BC$ respectively, not on the side extensions, and so we indeed
get a legitimate triple $P$, $Q$, $R$.
It remains to show that with this choice of points $P$, $Q$, $R$,
desired conditions on ratios of segments are satisfied, which will finish the proof.

Direct calculations show that the ratios $BR:BC$, $QR:PR$, $PC:AC$ are close to $(1+t)/2$
with accuracy $\delta$, and therefore differ pairwise by at most $2\delta$. Also, since
$\delta=\sqrt{\e}/8\leq 1/8$ and $-1/2\leq t\leq 1/2$, it follows that all these ratios are in range
from $(1-\frac{1}{2})/2- \frac{1}{8}$
to $(1+\frac{1}{2})/2+\frac{1}{8}$,
and therefore lie between $1/8$ and $1$. Similarly, the ratios $RC:BC$, $PQ:PR$, $AP:AC$ are close to $(1-t)/2$ with accuracy $\delta$,
so they also differ pairwise by at most $2\delta$ and lie between $1/8$ and $1$, which finishes the proof.
\end{proof}

\begin{rem}
Constant $\frac{1}{8}$ in the second part of Lemma~\ref{geom} is not tight. For the purpose of this work we only need the area to be
$\gtrsim \sqrt{\alpha}S^{5/6}$ and do not elaborate on the tight value for the constant,
however this might be of independent interest.
\end{rem}


\section{A construction of some convex polygonal curves}

Throughout the paper, we adopt the following notation.  We write $f(n)\gtrsim g(n)$ to mean
that $f(n)$ grows asymptotically at least as fast as the function $g(n)$, as $n$ tends to infinity.
Formally, if $f(n)$ and $g(n)$ are two non-negative
functions on integers, $f(n)\gtrsim g(n)$ if and only if there exist some positive constants $C$ and $N$, such that for each $n>N$, we have $f(n)\geq C g(n)$.

In this section we first give an overview of the ideas that are used in the proof of the main theorem, and we begin by
 constructing some polygonal curves whose limit will later be our strictly convex bounded curve $\gamma$.
We define inductively a sequence of convex $(AB,C)$-chains $\gamma_n$.
We start with the chain $\gamma_1=AB$.  Then at each step we modify $\gamma_n$ to $\gamma_{n+1}$ by appending one new point $Q$ to it, where Q comes from our pseudo-lattice, in such a way that $\gamma_{n+1}$ is a convex $(AB,C)$-chain and the probability that $Q$ is in $M_{q_n}$ is sufficiently high.

We also construct an auxiliary $(AB,C)$-chain $\gamma_n'$, such that at each step $\gamma_n$ is inscribed in $\gamma_n'$.
We will keep track of two quantities. The first quantity is $\ell_n$, the generalized affine perimeter of $\gamma_n$ with respect to $\gamma_n'$, defined by \eqref{genafflength}. It will be useful to prove that under some conditions $\ell_n$ is decreasing, but is bounded by positive constant from below. The second quantity is $s_n$, the area of locus of points $Q$ lying between $\gamma_n$ and $\gamma_n'$, such that these points can be added to $\gamma_n$ on the $n$'st step while not causing $\ell_n$ to decrease too rapidly.
Finally, we will estimate $s_n$ from below, as well as show that $s_n$ is decreasing. This lower bound  on $s_n$ will give us a sufficiently high probability for choosing a new point $Q$ from $M_{q_n}$ at the $n$'th step of our construction, for each $n\geq N$.

We now begin the construction of $\gamma_{n+1}$ from $\gamma_n$, while emphasizing the importance of choosing $Q$ appropriately.
We start with a pair of $(AB;C)$-chains $\gamma_1=AB$ and $\gamma_1'=ACB$.
At step~$n$ we have a pair of $(AB,C)$-chains $\gamma_n$ and $\gamma_n'$, $\gamma_n$ being inscribed in $\gamma_n'$,
\begin{align*}
&\gamma_n=AC_1C_2\dots C_{n-1}B,\\
&\gamma_n'=AD_1D_2\dots D_{n}B,
\end{align*}
where $C_i$ lies on $D_iD_{i+1}$ for each $i$.
We also denote triangles $\Delta_i=C_{i-1}D_iC_i$, $i=1,2,\dots,n$, where $C_0=A$, $C_n=B$, and their (doubled) areas as $S_i=S(\Delta_i)$.
Finally, we denote the affine length of $\gamma_n$ w.r.t. $\gamma_n'$ as
\begin{equation}\label{afflength}
\ell_n=l_A(\gamma_n:\gamma_n')=
\sum_{i=1}^{n} S_i^{1/3}.
\end{equation}

\begin{figure}
\centering
\label{fig:gamma}
\def\svgwidth{200pt}
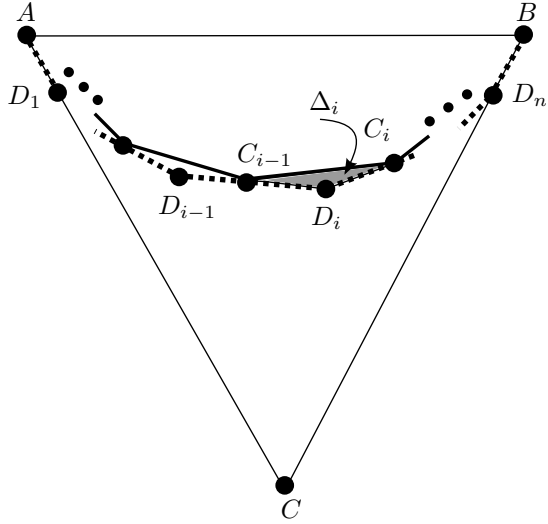
\caption{A triangle $ABC$ with two convex $(AB,C)$-chains $\gamma_n$ (bold) and $\gamma_n'$ (dashed bold) inside.  At the $i$'th step, a small triangle $\Delta_i$ is constructed.}
\end{figure}

Suppose the point $Q$ is chosen inside a triangle $\Delta_i$ and points $P$ and $R$ are chosen respectively on the line segments $C_{i-1}D_i$, $D_iC_i$,
such that $P$, $Q$, $R$ are collinear. Then we define a new pair of convex $(AB,C)$-chains $\gamma_{n+1}$ and $\gamma_{n+1}'$ as follows:
\begin{align*}
\gamma_{n+1}&=AC_1\dots C_{i-1}QC_i\dots C_{n-1}B,\\
\gamma_{n+1}'&=AD_1\dots D_{i-1}PRD_{i+1}\dots D_nB.
\end{align*}

A point $Q \in \bigcup_{i=1}^{n} \Delta_i$ is called  {\bf $\alpha$-\textit{admissible}} if
$\ell_{n}-\ell_{n+1}\leq \alpha$, where $\ell_{n+1}$ is now defined according to \eqref{afflength} using $\gamma_{n+1}$ and $\gamma_{n+1}'$.
We note that this definition is consistent with the definition given in the second part of Lemma~\ref{geom}.

The following Lemma shows that $\ell_n$ is bounded from below.

\bigskip
\begin{lemma}\label{lm:lowerbound}

Suppose a sequence $\ell_n$ has the recursion property that $\ell_n-\ell_{n+1}\leq a_n$, where the sequence $a_n$ is defined by
\begin{equation*}
a_n =
\begin{cases}
 \frac{1}{2} \ell_n  &   \text{if } n=0,1, {\text or } \ 2
\\
2\ell_n n^{-1}(\log n)^{-3/2} & \text{if } n \geq 3.
\end{cases}
\end{equation*}

Then $\ell_n\gtrsim \ell_0$ .
\end{lemma}
\begin{proof}

Indeed, for each $0\leq k\leq n-1$ we have $\ell_{k}-\ell_{k+1}\leq a_k$, and with formula for $a_k$, we can rewrite this inequality as
$$
\ell_{k}-\ell_{k+1}\leq 2\ell_{k} k^{-1}(\log k)^{-3/2},
$$
which can be again rewritten as
$$
\ell_{k+1}\geq \ell_{k}\left(1-\frac{2}{ k\log^{3/2} k}\right).
$$
Multiplying such inequalities for each $0\leq k\leq n-1$, we get
$$
\ell_{n} \geq \ell_0\prod_{k=0}^{n-1}\left(1-\frac{2}{ k\log^{3/2} k}\right).
$$
Let us denote the product on a right hand side as $X$. In order to show $\ell_n\gtrsim \ell_0$, it is now sufficient to show that $X$ is bounded by a positive constant
from below. Since $0\leq X\leq 1$, this is equivalent to showing that $\log X$ is bounded by some constant from below.
We can expand $\log X$:

\begin{align*}
\log X=
\log \prod_{k=0}^{n-1}\left(1-\frac{2}{ k\log^{3/2} k}\right)=
\sum_{k=0}^{n-1}\log\left(1-\frac{2}{ k\log^{3/2} k}\right)\geq
\sum_{k=0}^{n-1}\left(-\frac{1}{ k\log^{3/2} k}\right).
\end{align*}
The last inequality is true due to the fact that $\log(1-2x)$ is greater than $-x$ for all $0\leq x\leq 0.5$, and the fact that for
each $k>0$  all of the real numbers
$\frac{1}{ k\log^{3/2} k}$ lie in $[0, 0.5]$.
Thus $\log X$ is bounded from below, because the corresponding series
$\sum_{k=0}^{n-1}\left(\frac{1}{ k\log^{3/2} k}\right)$
converges, and the proof of the Lemma is complete.

\end{proof}

If for each $1\leq i\leq n$, the $i$'th step in our process has a corresponding point $Q$ which is $a_i$-admissible, where $a_i$ is defined as in Lemma~\ref{lm:lowerbound}, then by the conclusion of  Lemma~\ref{lm:lowerbound} we have
\begin{equation}\label{length_estim}
\ell_n\gtrsim \ell_0.
\end{equation}

\begin{figure}
\begin{center}
\label{testing}
\includegraphics[scale=0.30]{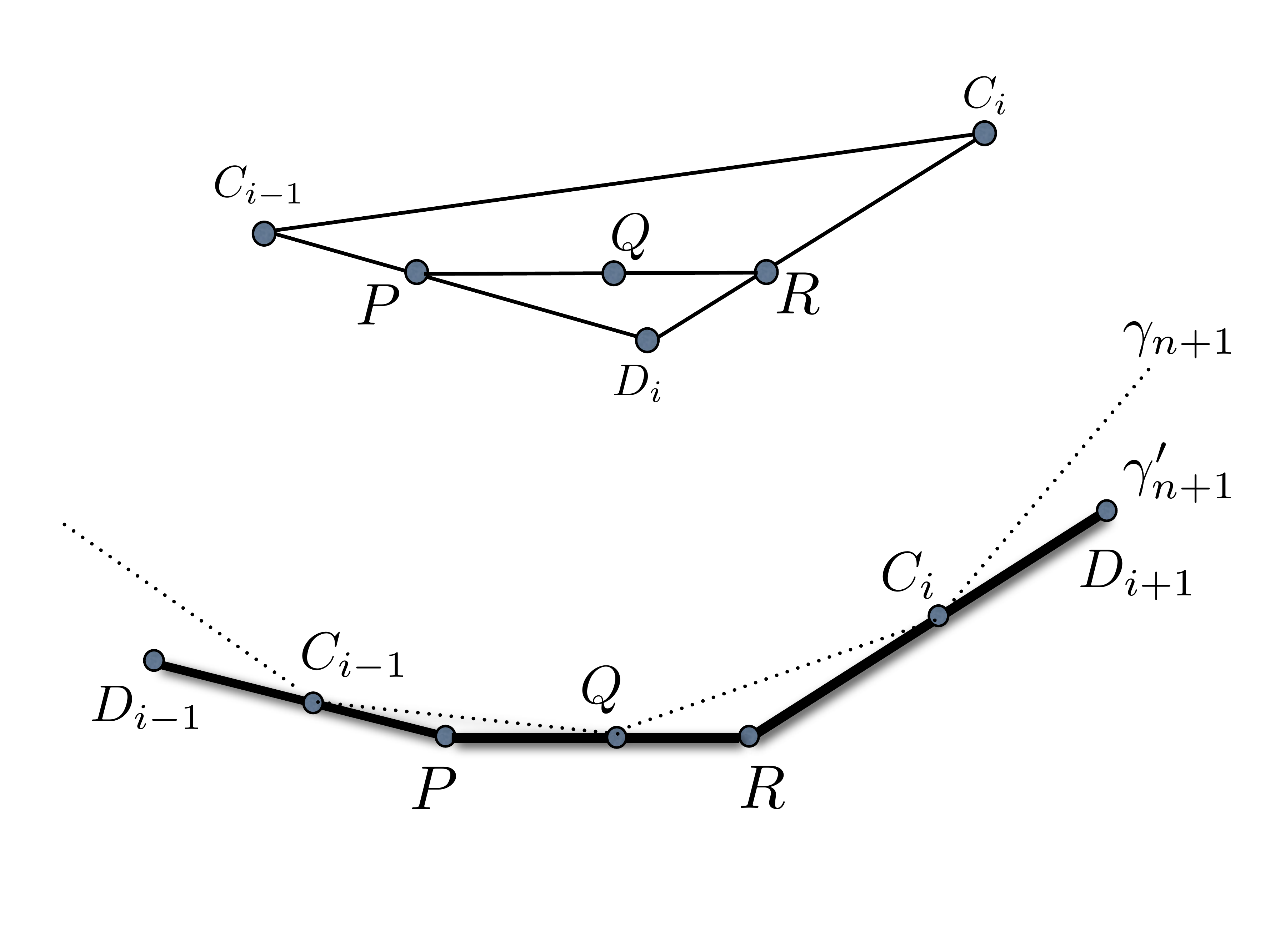}
\end{center}
\caption{At the $n$'th step of our inductive construction, we choose an $a_n$-admissible point $Q$ (above) inside a
 little triangle $\Delta_i$, and then we use the Poisson-point $Q$ to construct two new convex chains $\gamma_{n+1}$ and
 $\gamma_{n+1}'$.
 }
\end{figure}

Now we specify the choice of a point $Q$ at the $n$'th step of our inductive process. We choose $Q$ from
$\bigcup_{i=1}^{n} \Delta_i$ by using the following recipe. If there exists a point of a Poisson process $M_{q_n}$ that is $a_n$-admissible, then we let $Q$ be any such point. Otherwise, we let $Q$ be any $a_n$-admissible point.

We define $s_n$ to be the measure (area) of the set of all points $Q \in \bigcup_{i=1}^{n} \Delta_i$, such that $Q$ is
$a_n$-admissible. Next, we estimate $s_n$ from below.

\bigskip

\bigskip

\begin{lemma}
\begin{equation}\label{smallarea}
 s_n\gtrsim \ell_0^3 n^{-2}\log^{-3/4}n.
\end{equation}
\end{lemma}
\begin{proof}
To prove this, we start from partitioning the set of indices $\{ 1,2,\dots,n \}$
into two sets $G$ and $B$ (good and bad), where $i\in B$ if $S_i^{1/3}\leq a_n$ and $i\in G$ if $S_i^{1/3}> a_n$.
Note that $\sum_{i\in B} S_i^{1/3}\leq na_n\leq \ell_n/2$ by choice of $a_n$,
so
\begin{equation}\label{good_sum}
\sum_{i\in G} S_i^{1/3}\geq \ell_n/2.
\end{equation}
For any good index $i$ (assume $|G|=k$, so there are $k$ good indices)
we can apply Lemma\ref{geom},(ii), to triangle $\Delta_i$:
the set of $a_n$-admissible points in $\Delta_i$ has area at least
$$
\frac{1}{8}\sqrt{a_n}S_i^{5/6}
$$
Then the total area $s_n$ of the set of $a_n$-admissible points can be estimated as follows:
\begin{align*}
s_n\geq \sum_{i\in G} \frac{1}{8}\sqrt{a_n}S_i^{5/6}=
&\frac{1}{8}\sqrt{a_n}k\frac{\sum_{i\in G} S_i^{5/6}}{k}\gtrsim\\
&\sqrt{a_n}k\left(\frac{\sum_{i\in G} S_i^{1/3}}{k}\right)^{\frac{5}{6}:\frac{1}{3}}\geq\\
&\sqrt{a_n}k\left(\frac{\ell_n}{2k}\right)^{5/2}\gtrsim \\
&\sqrt{a_n}\ell_n^{5/2}k^{-3/2}\geq\\
&\sqrt{a_n}\ell_n^{5/2}n^{-3/2}=\\
&{(C_0\ell_n n^{-1}\log^{-3/2} n)}^{1/2}\cdot\ell_n^{5/2}n^{-3/2}\gtrsim\\
&\ell_n^3 n^{-2}\log^{-3/4}n\gtrsim\\
&\ell_0^3 n^{-2}\log^{-3/4}n.
\end{align*}
(our first inequality follows from power mean estimate of $\sum S_i^{5/6}$ via
$\sum S_i^{1/3}$, second inequality follows from \eqref{good_sum}, and the last inequality follows from \eqref{length_estim}).
\end{proof}


\begin{rem}
$s_n$ is bounded from above by the area of $ABC$.
\end{rem}

\begin{lemma}\label{lm:zeroprob}
Probability that a given triangle $\Delta$ is never divided in our process is $0$.
\end{lemma}
\begin{proof}
Assume that triangle $\Delta$ is never divided after step $N$, we want to find the probability of such event. We note that area $A$ of $\Delta$ remains constant throughout all the steps $n>N$.
Probability that on step $n$ the point $Q$ is chosen from a triangle $\Delta_i$ equals $\frac{A_n}{s_n}$, where $A_n$ is the contribution of $\Delta$ to $s_n$.
So the probability that $\Delta$ is never divided after $N$ steps equals
\begin{equation}
\prod_{n>N}\left(1-\frac{A_n}{s_n}\right).
\end{equation}
Showing that this product is $0$ is equivalent to showing that its inverse is infinity. But the inverse expression can be estimated from below as follows:
\begin{align}
&\prod_{n>N}\left(\frac{1}{1-\frac{A_n}{s_n}}\right) \geq
\prod_{n>N}\left(1+\frac{A_n}{s_n}\right) \geq
\sum_{n>N}\frac{A_n}{s_n} \geq \\
&\sum_{n>N}\frac{\frac{1}{8}\sqrt{a_n}A^{5/6}}{S(ABC)} \geq
\sum_{n>N}C\sqrt{n^{-1}\log^{-3/2}n},
\end{align}
where $C$ is some positive constant. The series in the last expression diverges, which ends the proof.
\end{proof}

The following geometrical lemma is a technical argument that is required to prove that length of segments of $\gamma_n$ approaches $0$ with probability $1$.
\begin{lemma}\label{lm:lengthbound}
Suppose we are in conditions of Lemma~\ref{geom}, and the additional property is satisfied: ratios $AP:AC$, $PQ:PR$, $BR:BC$ lie between $1/8$ and $7/8$. Let $m$
be the maximum of lengths of the sides of $ABC$. Then lengths of all segments $AP$, $PQ$, $QR$, $RB$, $AQ$, $QB$ are at most $\frac{399}{400} m$.
\end{lemma}
\begin{proof}
From the bounds on ratios it is immediate that lengths of $AP$, $PQ$, $QR$, $RB$ are at most $\frac{1}{8}m$. However estimating $AQ$ and $QB$ needs some more work.
Let $D$ be the point of intersection of $AQ$ and $BC$. We apply Menelaus' theorem to triangle $ACD$ and points $P$, $Q$, $R$:
\begin{equation}\label{eq:menel1}
\frac{AQ}{QD}\cdot\frac{DR}{RC}\cdot\frac{CP}{PA}=1,
\end{equation}
and to triangle $PCR$ and points $A$, $Q$, $D$:
\begin{equation}\label{eq:menel2}
\frac{PA}{AC}\cdot\frac{CD}{DR}\cdot\frac{RQ}{QP}=1.
\end{equation}
From \eqref{eq:menel2} we get
\begin{equation}\label{eq:menel3}
\frac{CD}{DR}=\frac{AC}{PA}\cdot\frac{QP}{RQ}\leq 8\cdot 7=56.
\end{equation}
Now combining \eqref{eq:menel1} with \eqref{eq:menel3} we get
\begin{equation*}
\frac{AQ}{QD}=\frac{RC}{DR}\cdot\frac{PA}{CP}=(\frac{CD}{DR}+1)\cdot\frac{PA}{CP}\leq (56+1)\cdot 7=399,
\end{equation*}
which finally gives us
\begin{equation*}
AQ\leq 399QD\Rightarrow 400AQ\leq 399(AQ+QD)\Rightarrow AQ\leq \frac{399}{400}AD\Rightarrow AQ\leq \frac{399}{400}m.
\end{equation*}
Length of $QB$ can be estimated similarly.
\end{proof}

\begin{lemma}
Probability that $\gamma$ contains a straight line segment is $0$.
\end{lemma}
\begin{proof}
We will first show that lengths of all segments $C_{i-1}D_i$, $D_iC_i$ and $C_{i-1}C_i$ approach $0$ with probability $1$.
We can additionally require that $\alpha$-admissible points $Q$ satisfy the following property: ratios $AP:AC$, $PQ:PR$, $BR:BC$ lie between $1/8$ and $7/8$.
We can assume that, because the set of $\alpha$-admissible points constructed in the proof of Lemma~\ref{geom} satisfies this property.
Our statement immediately follows from combining lemmas \ref{lm:zeroprob} and \ref{lm:lengthbound}.

Now assume there is a straight line segment $l$ in $\gamma$. Consider a small ball $B$ with the center
in the midpoint of this straight line segment. Lengths of all segments of $\gamma_n$ approach $0$, while $\gamma_n$ themselves approach $\gamma$, so there should exist an infinite sequence of vertices $Q_k$ of $\gamma_{n_k}$ inside the ball $B$. By convexity of $\gamma$, all $Q_k$ should lie on $l$, which contrdicts the fact that all $\gamma_n$ are strictly convex.
\end{proof}

\begin{lemma}\label{lm:main}
Fix any triangle $ABC$ in the plane, and fix any $\e>0$.    Then with probability at least $1-\e$,
 there exists
a convex countable $(AB,C)$-chain $\gamma$ with at least one point from each pseudo-lattice $M_{q_n}$, for some $N$ with  $n \geq N$.
\end{lemma}
\begin{proof}
Consider now a Poisson random
configuartion $M_{q_n}$ on the plane. If $q_n$ is a power of prime $p$, then intensity is $q_n^2\cdot (1-\frac{1}{p^2})$ points
per unit area. Also we note that $q_n\gtrsim n\log n$, and $1-\frac{1}{p^2}\geq \frac{1}{2}$.
Denote the probability that there are no $a_n$-appropriate points w.r.t. $\gamma_n$ by $P_n$.
We can estimate $P_n$ from above using \eqref{smallarea}:
\begin{align}\label{smallprob}
P_n=\exp(-q_n^2(1-\frac{1}{p^2})s_n)
& \leq  \exp\left(-n^2\log^{2}n \cdot C\ell_0^{3}n^{-2}\log^{-3/4}n\right)\notag \\
&=\exp\left(-C\ell_0^{3}\log^{5/4}n\right)\notag \\
&= n^{-C\ell_0^{3}\log^{1/4}n},
\end{align}
where $C$ stands for a positive constant.

By \eqref{smallprob} the series $\sum{P_n}$ does converge, which can be formally written as follows: for any $\e>0$ there exists positive constant $N$, such that
$$\sum_{n\geq N}{P_n}<\e.$$
It means that for any $\e>0$ and for any triangle $ABC$ one may find a big constant
$N$ such that the probability to find a convex countable $(AB,C)$-chain with at least one point
from each pseudo-lattice $L_{q_n}, n\geq N$ is not less than $1-\e$.

\end{proof}

\begin{proof}[Proof of Theorem\ref{tm:main}]
Fix $\e_0>0$.
Consider the countable sequence of points $A_1A_2\dots$ on a circle, monotonically
converging to some point $A_0$. Consider triangles, formed by chords $A_iA_{i+1}$
and tangents in its endpoints. Apply Lemma~\ref{lm:main} for $i$-th triangle taking
$\e_i=\e_0/2^i$. We get a series of convex countable chains $\gamma^i$, each of which has at least
one point from each $L_{q_n}, n\geq N_i$ with probability at least $\e_0/2^i$.
 Note that the events in different triangles are independent,
so the processes inside them do not depend on what have we found in other triangles.
Finally we take $\gamma$ as a union of $\gamma^i$ over all $i$, note that $\gamma$ is convex by the choice of corresponding triangles.
By construction, $\gamma\cap L_{q_n}\rightarrow\infty$  with probability at least $\e_0$.
\end{proof}

\section{Questions}

There are many theorems and open questions concerning
rational points on convex curves and surfaces, which may be
settled for pseudo-lattices as well. For instance, consider the
analogue of the main result of \cite{P1}:   if $\gamma$ is a bounded strictly convex curve in the
Euclidean plane, then $|\gamma\cap \frac1n \mathbb{Z}^2|=o(n^{2/3})$.

The corresponding conjecture for pseudo-lattices is that
with probability $1$ for any sequence of independent
Poisson spatial processes $M_d$ with intensities $d=1,2,\dots$,
the asymptotic bound $|\gamma\cap M_d|=o(d^{2/3})$ holds
for any strictly convex curve $\gamma$.

At the moment, we do not even know whether such a probability measure exists,
and the reason for the difficulty here is that $\gamma$ is arbitrary, which makes it hard to prove measurability.

\end{document}